\documentclass[12pt,a4paper]{amsart}
\usepackage{graphics}
\usepackage{epsfig}
\usepackage{graphicx}
\theoremstyle{plain}
\usepackage{amssymb}
\usepackage[T2A]{fontenc}
\usepackage[cp1251]{inputenc}
\usepackage[ukrainian, english]{babel}
\advance\hoffset-20mm \advance\textwidth40mm

\newtheorem{theorem}{Theorem}
\newtheorem{lemma}{Lemma}
\newtheorem*{theo*}{Theorem}

\newtheorem{corollary}{Corollary}
\theoremstyle{definition}

\newtheorem*{definition*}{Definition}

\newtheorem{remark}{Remark}

\begin{document}
\sloppy
\title[On  nilpotent Lie algebras of derivations of fraction fields]
{On nilpotent   Lie algebras \\ of   derivations of fraction fields}
\author
{ A.P.  Petravchuk}
\address{Anatoliy P. Petravchuk:
Department of Algebra and Mathematical Logic, Faculty of Mechanics and Mathematics, Kyiv
Taras Shevchenko University, 64, Volodymyrska street, 01033  Kyiv, Ukraine}
\email{aptr@univ.kiev.ua , apetrav@gmail.com}
\date{\today}

%\thanks{
% The second author was partially supported by the DFFD,
%grant F28.1/026}
%
\begin{abstract}
Let $\mathbb K$ be an arbitrary  field of characteristic zero and $A$  an integral $\mathbb K$-domain. Denote by $R$ the fraction field of $A$ and by $W(A)=RDer_{\mathbb K}A,$ the Lie algebra of  $\mathbb K$-derivations on $R$ obtained from $Der_{\mathbb K}A$ via multiplication by elements of $R.$  If $L\subseteq W(A)$ is a subalgebra of $W(A)$ denote by $rk_{R}L$ the dimension of the vector space $RL$ over the field  $R$ and by $F=R^{L}$ the field of constants of $L$ in $R.$  Let $L$ be a  nilpotent subalgebra $L\subseteq W(A)$ with  $rk_{R}L\leq 3$. It is proven that the Lie algebra  $FL$ (as a Lie algebra over the field $F$) is isomorphic to a finite dimensional  subalgebra of the triangular  Lie subalgebra $u_{3}(F)$ of the Lie algebra  $Der F[x_{1}, x_{2}, x_{3}], $ where  $u_{3}(F)=\{f(x_{2}, x_{3})\frac{\partial}{\partial x_{1}}+g(x_{3})\frac{\partial}{\partial x_{2}}+c\frac{\partial}{\partial x_{3}}\}$ with $f\in F[x_{2}, x_{3}], g\in F[x_3]$, $c\in F.$ In particular, a characterization of nilpotent Lie algebras of vector fields with polynomial coefficients in three variables is obtained.
 \end{abstract}
\maketitle

%%%%%%%%%%%%%%%%%%%%%%%%%%%%%%%%%%%%%%%%%%%%%%%%%%%%%%%%%%%

\section{Introduction}
Let $\mathbb{K}$ be an arbitrary   field of characteristic zero and $A$  an associative commutative $\mathbb K$-algebra that is a domain.
The set  ${\rm Der} _{\mathbb K}A$  of all $\mathbb K$-derivations of $A$   is a Lie algebra over $\mathbb K$  and an $A$-module in a natural way:  given $a\in A, D\in {\rm Der} _{\mathbb K}A,$   the derivation $aD$ sends any element $x\in A$ to $a\cdot D(x).$
The structure of the Lie algebra ${\rm Der} _{\mathbb K} A$ is of great interest because in case $\mathbb K=\mathbb R$ and $A=\mathbb R[[x_{1}, \ldots , x_{n}]],$ the   ring of formal power series, the Lie algebra of all $\mathbb K$-derivations of the form
$$D=f_{1}\frac{\partial}{\partial x_{1}}+\cdots +f_{n}\frac{\partial}{\partial x_{n}}, \ f_{i}\in \mathbb R[[x_{1}, \ldots , x_{n}]]$$ can be considered as the Lie algebra of vector fields on $\mathbb R^{n}$ with formal power series coefficients.
Such Lie algebras   with polynomial, formal power series, or analytical coefficients were  studied by many authors. Main results for fields $\mathbb K=\mathbb C$ and $\mathbb K=\mathbb R$ in case  $n=1$ and $n=2$ were obtained in \cite{Lie} \cite{Olver},   \cite{Olver1} (see also \cite{Bavula1}, \cite{Draisma},  \cite{PBNL}, \cite{Post}).

One of the important problems in Lie theory is to describe finite dimensional subalgebras of the Lie algebra $\overline{W_{3}}(\mathbb C)$ consisting of all derivations on the ring $\mathbb C[[x_{1}, x_{2}, x_{3}]]$ of the form $$a_{1}\frac{\partial}{\partial x_{1}}+a_{2}\frac{\partial}{\partial x_{2}}+a_{3}\frac{\partial}{\partial x_{3}}, a_{i}\in \mathbb C[[x_{1}, \ldots , x_{n}]].$$  In order to characterize  nilpotent subalgebras of the Lie algebra $\overline{W_{3}}(\mathbb C)$ we consider more general situation. Let $R={\rm Frac}(A) $ be the field of fractions of an integral domain $A$ and $W(A)=R{\rm Der}_{\mathbb K}(A)$ the Lie algebra of derivations of the field $R$ obtained from derivations on $A$ by multiplying by elements of the field  $R$ (obviously ${\rm Der} _{\mathbb K}A\subseteq W(A)$). For a subalgebra $L$ of the Lie algebra $W(A)$ let us define ${\rm rk}_{R}(L)=\dim _{R}RL$  and denote by $F=R^{L}=\{ r\in R \ | \ D(r)=0,\forall D\in L\}$   the field of constants of the Lie algebra $L.$ The $\mathbb K$-space $FL$ is  a vector space over the field $F$ and   a Lie algebra  over  $F.$ If  $L$ is a nilpotent subalgebra of $W(A),$ then $FL$ is finite dimensional over $F$ (by Lemma \ref{nilfewdim}).

The main result of the paper: If $L$ is a nilpotent subalgebra of rank $k\leq 3$ over $R$ from the Lie algebra $W(A),$
then $FL$ is isomorphic to a finite dimensional subalgebra of the triangular Lie algebra $u_{k}(F)$ (Theorem \ref{th2}). Triangular Lie algebras were studied in \cite{Bavula1} and \cite{Bavula2}, they are locally nilpotent but not nilpotent, the structure of their  ideals was described in these papers.

We use standard notation, the ground field is arbitrary of characteristic zero. The quotient field of the integral domain $A$ under consideration is denoted by $R.$ Any derivation $D$ of $A$ can be uniquely extended  to a derivation of $R$ by the rule: $D(a/b)=(D(a)b-aD(b))/b^{2}.$
If $F$ is a subfield of the field $R$ and $r_{1}, \ldots , r_{k}\in R, $ then the set of all linear combinations of these elements with coefficients in $F$ is denoted by $F\langle r_{1}, \ldots , r_{k}\rangle , $ it is a  subspace of the $F$-space $R.$
The triangular subalgebra $u{_n}(\mathbb K)$ of the Lie algebra $W_{n}(\mathbb K)={\rm Der} (\mathbb K[x_1, \ldots , x_n])$  consists of all the derivations on the ring $\mathbb K[x_1, \ldots , x_n]$ of the form $D=f_{1}(x_2, \ldots x_n)\frac{\partial}{\partial x_1}+\cdots +f_{n-1}(x_{n})\frac{\partial}{\partial x_{n-1}}+f_{n}\frac{\partial}{\partial x_n},$  where $f_{i}\in \mathbb K[x_{i+1}, \ldots x_n], f_n\in \mathbb K.$

\section{Some properties of nilpotent subalgebras of $W(A)$}
We will use  some statements about derivations and nilpotent Lie algebras of derivations from  the paper  \cite{MP1}. The next statement can be immediately checked.
\begin{lemma}\label{commutators}
Let $D_{1}, D_{2}\in W(A)$  and $a, b\in R.$ Then it holds:

{\rm 1.}  $[aD_{1}, bD_{2}]=ab[ D_1, D_2]+aD_1(b)D_2-bD_2(a)D_{1}.$

{\rm 2.} If $a, b\in R^{D_{1}}\cap R^{D_{2}},$  then  $[aD_{1}, bD_{2}]=ab[ D_1, D_2].$
\end{lemma}

Let $L$ be a subalgebra of rank $k$ over $R$ of the Lie algebra $W(A)$
and $F=R^{L}$ its field of constants. Denote by $RL$ the set of all linear combinations over $\mathbb K$ of elements $aD,$ where $a\in R$ and $D\in L.$  The set $FL$  is defined analogously.

\begin{lemma}(\cite{MP1}, Lemma 2).\label{MP-2}
Let $L$ be a nonzero subalgebra of  $W(A)$ and let $FL$, $RL$ be  $\mathbb{K}$-spaces  defined as above.
Then:

{\rm 1.}  $FL$ and $RL$ are $\mathbb{K}$-subalgebras of the Lie algebra $W(A)$.
Moreover,  $FL$ is a Lie algebra over the field $F.$

{\rm 2.} If the algebra $L$ is  abelian,  nilpotent, or  solvable  then the Lie algebra $FL$ has the same property, respectively.
\end{lemma}

\begin{lemma}(\cite{MP1}, Lemma 3).\label{MP-3}
Let $L$ be a subalgebra of finite rank over $R$  of the Lie algebra $W(A)$, $Z=Z(L)$  the center of $L,$
and $F=R^{L}$  the field of constants of $L$. Then ${\rm rk}_{R} Z = \dim _{F}FZ$
and $FZ$ is a subalgebra of the center $Z(FL)$. In particular, if
$L$ is abelian, then $FL$ is an abelian subalgebra  of  $W(A)$ and ${\rm rk}_{R}L=\dim _{F}FL.$

\end{lemma}

\begin{lemma} (\cite{MP1}, Lemma 4).\label{MP-4}
Let $L$ be a  subalgebra of  the Lie algebra $W(A)$ and  $I$ be an ideal of $L$. Then
the vector  space $RI\cap L$ (over $\mathbb K$)  is also  an ideal  of $L.$
\end{lemma}

\begin{lemma}(\cite{MP1}, Proposition 1, Theorem 1). \label{nilfewdim}
Let $L$ be a nilpotent subalgebra of $W(A)$ and $F=R^{L}$ be its field of
constants. Then:

{\rm 1.}  If ${\rm rk}_{R}L<\infty ,$  then $\dim_{F}FL<\infty .$

{\rm 2.}  If ${\rm rk}_{R}L=1,$ then $L$ is abelian and $\dim_{F}FL=1.$

{\rm 3.}  If ${\rm rk}_{R}L=2,$ then there exist elements $D_1, D_2 \in FL$ and $a \in R$ such that
 $$FL=F\langle D_1, aD_1, \ldots , \frac{a^k}{k!} D_{1}, D_{2}\rangle , \  k\geq 0  \ (\mbox{if}$$
 $$ \ k=0, \ \mbox{then put} \ FL=F\langle D_{1}, D_{2}\rangle ), $$
where $[D_1, D_2]=0$, $D_1(a)=0$, $D_2(a)=1.$
\end{lemma}

\begin{lemma}\label{inclusion1}
Let $D_{1}, D_{2}, D_{3}\in W(A) $ and $a\in R$ be such elements that $D_{1}(a)=D_{2}(a)=0,$  $D_{3}(a)=1$  and let $F=\cap _{i=1}^{3}R^{D_{i}}.$ If there exists an element $b\in R$ such that
$D_{1}(b)=D_{2}(b)=0, \ D_{3}(b)\in F\langle 1, a, \ldots , a^{s}/s!\rangle $ for some $s\geq 0,$  then $b\in F\langle 1, a, \ldots , a^{s+1}/(s+1)!\rangle .$
\end{lemma}
\begin{proof}
Write down $D_{3}(b)=\sum _{i=0}^{s}\beta _{i}a^{i}/i!$ with $\beta _{i}\in F$ and take the element $c=\sum _{i=0}^{s}\beta _{i}a^{i+1}/(i+1)!$ of the field $R.$ It holds obviously  $D_{3}(b-c)=0$ and (by the conditions of Lemma)  $D_{1}(b-c)=0$ and $D_{2}(b-c)=0.$ Then we have $b-c\in \cap _{i=1}^{3}R^{D_{i}}=F,$ and therefore $b=\gamma +\sum _{i=0}^{s}\beta _{i}a^{i+1}/(i+1)!$ for some element $\gamma \in F.$ The latter means that $b\in F\langle 1, a, \ldots , a^{s+1}/(s+1)!\rangle .$

\end{proof}
\begin{lemma}\label{inclusion2}
Let $D_{1}, D_{2}, D_{3}\in W(A) $ and $a, b \in R$ be such elements that $D_{1}(a)=D_{1}(b)=0,$   $D_{2}(a)=1$ $D_{2}(b)=0,$ $D_{3}(a)=0,$ $ D_{3}(b)=1$ and let  $F=\cap _{i=1}^{3}R^{D_{i}}.$ If  there exists  an element $c\in R$ such that  $D_{1}(c)=0, [D_{2}, D_{3}](c)=0,$  $D_{2}(c)\in F\langle \{\frac{a^{i}b^{j}}{i!  j!}\}, 0\leq i\leq m-1, 0\leq j\leq k\rangle ,$  \ $D_{3}(c)\in F\langle \{\frac{a^{i}b^{j}}{i! j!}\}, 0\leq i\leq m, 0\leq j\leq k-1\rangle ,$ then $c\in F\langle \{\frac{a^{i}b^{j}}{i! j!}\}, 0\leq i\leq m, 0\leq j\leq k\rangle $

\end{lemma}

\begin{proof}
 The elements $D_{2}(c)$ and $D_{3}(c)$ can be written (by conditions of the lemma) in the form $D_{2}(c)=f(a, b), D_{3}(c)=g(a, b)$ where $f, g\in F[u, v]$ are some  polynomials of $u, v.$  Since $[D_{2}, D_{3}](c)=0,$ it holds $D_{2}(g)=D_{3}(f).$  It follows from the relations   $D_{2}(g)=\frac{\partial}{\partial a}g(a, b), D_{3}(f)=\frac{\partial}{\partial b}f(a, b)$ that  $ \frac{\partial}{\partial a}g(a, b)=\frac{\partial}{\partial b}f(a, b).$  Hence  there exists a polynomial $h(a, b)\in F[a, b]$ (the potential of the  vector field $f(a, b)\frac{\partial}{\partial a}+g(a, b)\frac{\partial}{\partial b}$) such that $D_{3}(h(a, b))=g,  D_{2}(h(a, b))=f.$ The polynomial $h(a, b)$ is obtained from the polynomials $f, g$ by formal integration on $a$ and on $b,$  so we have  $ h(a, b)\in F\langle  \{\frac{a^{i}b^{j}}{i! j!}\}, 0\leq i\leq m, 0\leq j\leq k\rangle .$ Further,  using  properties  of the element $h(a, b)$ we get  $D_{2}(h-c)=D_{3}(h-c)=0.$ Besides,  it  holds  $D_{1}(h-c)=0. $  The latter  means that  $h-c\in F=\cap _{i=1}^{3}R^{D_{i}}.$  But then $c=\gamma +h$ for some $\gamma \in F$ and therefore $c\in F\langle \{\frac{a^{i}b^{j}}{i! j!}\}, 0\leq i\leq m, 0\leq j\leq k\rangle .$
\end{proof}

\section{On nilpotent  subalgebras of small rank of $W(A)$}
\begin{lemma}\label{rk2}
Let $L$ be a nilpotent subalgebra of rank $3$ over $R$ from the Lie algebra $W(A)$ and $F=R^{L}$ be its field of constants. If the center $Z(L)$ of the algebra $L$ is of rank $2$ over $R$ and $\dim _{F}FL\geq 4,$  then there exist $D_1, D_2, D_3\in L, a\in R$  such that the Lie algebra $FL$ is contained in a nilpotent Lie algebra $ \widetilde L$ of the Lie algebra $W(A)$ of the form
$$\widetilde{L}=F\langle D_{3}, D_{1}, aD_{1}, \ldots, (a^{n}/n!)D_{1}, D_{2}, aD_{2}, \ldots , (a^{n}/n!)D_{2}\rangle $$
 for some $n\geq 1, $ with  $[D_{i}, D_{j}]=0, i,j=1,2,3, D_{1}(a)=D_{2}(a)=0, D_{3}(a)=1.$
\end{lemma}
\begin{proof}
Take any elements $D_{1}, D_{2}\in Z(L)$ that are linearly independent over $R$ and denote $I=(RD_{1}+RD_{2})\cap L.$  Then $I$ is an ideal of the Lie algebra $L$ (by Lemma \ref{MP-4}).  Take an arbitrary element $D\in I$ and write down $D=a_{1}D_{1}+a_{2}D_{2}$ for some elements $ a_{i}\in R.$ Since $[D_{i}, D]=0=D_{i}(a_1)D_{1}+D_{i}(a_2)D_{2}, i=1,2$ we get $D_{i}(a_{j})=0, i,j=1,2.$ It follows easily that for any  element $ D'\in I $ it holds the equality $[D, D']=0,$  so the ideal $I$ is abelian. The Lie algebra $FL$ is finite dimensional over $F$ and $\dim _{F}FL/FI=1$ by Lemma \ref{nilfewdim}. Take any element $D_{3}\in L\setminus I.$ Then $FL=FI+FD_{3}$ and $D_{1}, D_{2}, D_{3}$ are linearly independent over $R.$

Since ${\rm rk} _{R}Z(L)=2,$ (by  conditions of the lemma) we
 have \\ ${\rm dim} _{F}FZ(L)=2$ by Lemma \ref{MP-3}. The ideal $I$ of the Lie algebra $L$ is abelian by the above proven, so  the ideal $FI$ of the Lie algebra $FL$ over the field $F$ is also abelian.   Since  $FL=FI+FD_{3},$  there exists  a basis of the $F$-space $FI$ in which  the nilpotent  linear operator ${\rm ad} D_{3}$  has a matrix  consisting of two Jordan blocks. Let $J_{1}$ and $J_{2}$ be the correspondent Jordan bases; without loss of generality one can assume that $D_{1}\in J_{1}, D_{2}\in J_{2}$ and the elements $D_{1}, D_{2}$ are the first members of the bases $J_{1}$ and $J_{2}$ respectively.

If $\dim _{F}F\langle J_{1}\rangle =\dim _{F}F\langle J_{2}\rangle =1,$ then $FL=F\langle D_{3}, D_{1}, D_{2}\rangle $ is of dimension $3$ over $F$  which contradicts the conditions of  the lemma.   So, we may assume that  $\dim _{F}F\langle J_{1}\rangle \geq \dim _{F}F\langle J_{2}\rangle $ and  $\dim _{F}F\langle J_{1}\rangle =n+1, n\geq 1.$  Denote the elements of the basis $J_{1}$ by $ D_{1}, a_{1}D_{1}+b_{1}D_{2}, \ldots , a_{n}D_{1}+b_{n}D_{2},$  where the elements $a_{i}, b_{i}$ belong to  $R$  and put for convenience $a=a_{1}.$ Let us  prove by induction on $i$ that $a_{i}, b_{i}\in F\langle 1, a, \ldots , a^{i}/i!\rangle .$ If $i=1,$ then  $a_{1}=a\in F\langle 1, a\rangle$ by  definition. It follows from the relation $[D_{3}, a_{1}D_{1}+b_{1}D_{2}]=D_{1}=D_3(a_1)D_1+D_3(b_1)D_2$  that $D_{3}(b_{1})=0.$ Since $FI$ is abelian (by the above proven), we have $D_{1}(b_{1})=D_{2}(b_{1})=0.$ The latter means that $b_{1}\in F \subset F\langle 1, a\rangle .$

Further, the relation
$$[D_{3}, a_{i}D_{1}+b_{i}D_{2}]=a_{i-1}D_{1}+b_{i-1}D_{2}=D_3(a_i)D_1+D_3(b_i)D_2$$
 gives the equalities $D_{3}(a_{i})=a_{i-1}$ and $D_{3}(b_{i})=b_{i-1}.$ By the inductive assumption,  $a_{i-1}, b_{i-1}\in F\langle 1, a, \ldots , a^{i-1}/(i-1)!\rangle $ and taking into account the relations $D_{j}(a_{i})=D_{j}(b_{i} )=0, j=1,2$ (they hold because $FI$ is abelian) we get by Lemma \ref{inclusion1} that $a_{i}, b_i \in F\langle 1, \ldots , a^{i}/i!\rangle .$  The latter relation means that the $F$-subspace $F\langle J_{1}\rangle$ of $FI$ lies in the subalgebra $\widetilde L$ from the  conditions of the lemma.

Now let
$$J_{2}=\{ D_{2}, c_{1}D_{1}+d_{1}D_{2}, \ldots , c_{k}D_{1}+d_{k}D_{2}\}$$
 be a basis corresponding to the second Jordan block. The relation $[D_{3}, c_{1}D_{1}+d_{1}D_{2}]=D_{2}$ implies the equality $D_{3}(d_{1}) =1$ and therefore  $D_{3}(a-d_{1})=0. $ Since  $D_{1}(a-d_{1})=D_{2}(a-d_{1})=0,$  we get $a-d_{1}\in F,$ i.e. $d_{1}=a+\gamma$ for some $\gamma \in F.$ Applying  the above considerations   to the Jordan basis $J_{2}$ we obtain that $F\langle J_{2}\rangle \subset \widetilde L.$ But then the  Lie algebra $L$ is entirely contained in $\widetilde L.$
\end{proof}
\begin{lemma}\label{rk1}
Let $L$ be a nilpotent subalgebra of rank $3$ over $R$ from the Lie algebra $W(A)$ and $F=R^{L}$ be its field of constants. If the center  $Z(L)$ of the algebra $L$ is of rank $1$ over $R$  and $\dim _{F}FL\geq 4,$ then the Lie algebra $FL$ is contained  either in the nilpotent Lie algebra $ \widetilde L$ from the conditions of Lemma \ref{rk2}  or in a subalgebra $\overline{L}$ of $W(A)$ of the form
$${\overline{L}}=F\langle D_{3}, D_{2}, aD_{2}, \ldots , (a^{n}/n!)D_{2}, \{ \frac{a^{i}b^{j}}{i! j!}D_{1}\} , 0\leq i, j\leq m\rangle $$
where $ n\geq 0, m\geq 1, D_i \in L, [D_{i}, D_{j}]=0, $ for $ i,j=1,2,3,$  and $ a, b\in R $ such that $D_{1}(a)=D_{2}(a)=0, D_{3}(a)=1, D_{1}(b)=D_{3}(b)=0, D_{2}(b)=1.$
\end{lemma}
\begin{proof}
Take any nonzero element $D_{1}\in Z(L)$ and denote $I_{1}=RD_{1}\cap L.$ Then $I_{1}$  is an abelian ideal of the algebra Lie $L$ and ${\rm rk}_{R}I_{1}=1$  by Lemma \ref{MP-3}. Choose any nonzero element $D_{2}+I_{1}$
in the center of the quotient Lie algebra $L/I_{1}$ and denote  $I_{2}=(RD_{1}+RD_{2})\cap L.$ By the same Lemma \ref{MP-3},   $I_{2}$ is an ideal of the Lie algebra $L$ and ${\rm rk}_{R}I_{2}=2.$ Further, take any element $D_{3}\in L\setminus I_{2}.$ Since $\dim _{F}FL/FI_{2}=1$ by Lemma 5 from the paper \cite{MP1},  we have $FL=FI_{2}+FD_{3}.$

{\underline{Case 1.}} The ideal $I_{2}$ is abelian. Let us show that $FL$ is contained in the Lie algebra $\widetilde L$ from the conditions of Lemma \ref{rk2}. It is obvious that $FI_{2}$ is an abelian ideal of codimension $1$ of the Lie algebra $FL$ over the field $F.$ By Lemma \ref{MP-3},  ${\rm rk}_{R}Z(L)={\rm dim} _{F}FZ(L)$  and by the  conditions of the lemma, we see that ${\rm dim} _{F}FZ(L)=1.$  The linear operator ${\rm ad} D_{3}$ acts on the $F$-space $FI_{2}$ and ${\rm dim} _{F}{\rm Ker} ({\rm ad} D_{3})={\rm dim} _{F}FZ(L).$ Therefore ${\rm dim} _{F}{\rm Ker} ({\rm ad} D_{3})=1$ and there exists a basis  of $FI_{2}$ in which ${\rm ad} D_{3}$ has a  matrix in the form of  a single Jordan block. The same is true for the action of ${\rm ad }D_{3}$ on the vector space $FI_{1}$ (since $[D_{3}, I_{1}]\subseteq I_{1},$ the ideal $FI_{1}$ is invariant under ${\rm ad} D_{3}$). The subalgebra $FI_{1}+FD_{3}$  is of rank $2$ over $R.$ If ${\rm dim} _{F}FI_{1}>1,$ then the center of the Lie algebra $F_{1}+FD_{3}$ is of dimension $1$  over $F.$ By Lemma \ref{nilfewdim},
 there exists a Jordan basis in $FI_{1}$ of the form
$$\{ D_{1}, aD_{1}, \ldots , (a^{s}/s!)D_{1}\}, \ \mbox{where} \ s\geq 0, D_{3}(a)=1, [D_{3}, D_{1}]=0.$$
If $\dim _{F}FI_{1}=1,$ then $s=0$  and the desired basis of $F_{1}$ is of the form $\{ D_{1}\}$.

Let first $s>0.$
Since $(FD_{2}+FI_{1})/FI_{1}$ is a central ideal of the quotient algebra $FL/FI_{1},$ we have $[D_{3}, D_{2}]\in FI_{1}$ and  hence one can  write $[D_{3}, D_{2}]=\gamma _{0}D_{1}+\ldots +(\gamma _{s}a^{s}/s!)D_{1}$  for some $\gamma _{i}\in F.$ Taking $D_{2}-\sum _{i=0}^{s-1}(\gamma _{i}a^{i+1}/(i+1)!)D_{1}$ instead $D_{2}$ we may assume that $[D_{3}, D_{2}]=\gamma _{s}a^{s}/s!D_{1}. $  Note that $\gamma _{s}\not =0.$
Really, in the opposite case  $[D_3, D_2]=0$ and therefore $D_{2}\in Z(L).$  Then ${\rm rk}_{R}Z(L)=2$ which is impossible because of  the conditions of the lemma. After changing  $D_{3}$ by $\gamma _{s}^{-1}D_{3}$ we may assume that $[D_{3}, D_{2}]=(a^{s}/s!)D_{1}.$

Since the linear operator ${\rm ad} D_{3}$ has in a basis of the $F$-space $FI_{2}$ a matrix, consisting of a single Jordan block, the same is true for the linear operator ${\rm ad} D_{3}$ on the vector space $FI_{2}/FI_{1}.$  Let $\dim _{F} FI_{2}/FI_{1}=k$ and $\{ {\overline S_{1}}, \ldots ,  {\overline S_{k}}\}$  be a Jordan basis for ${\rm ad} D_{3}$ on $FI_{2}/FI_{1},$  where ${\overline S_{i}}=(c_{i}D_{1}+d_{i}D_{2})+FI_{1}, \ i=1, \ldots , k, c_{i}, d_{i}\in R.$  The representatives $c_{i}D_{1}+d_{i}D_{2}$ of the cosets ${\overline S_{i}}$ can be chosen in such a way that $[D_{3}, c_{i}D_{1}+d_{i}D_{2}]=c_{i-1}D_{1}+d_{i-1}D_{2}, i=2, \ldots , k$  and
$$   [D_{3}, c_{1}D_{1}+d_{1}D_{2}]=\sum _{i=0}^{s}\beta _{i}(a^{i}/i!)D_{1} \eqno (1)    $$
for some $\beta _{i}\in F.$  Let us show by induction on $i$ that the relations hold:
$$ d_{i}\in F\langle 1, \ldots , a^{i-1}/(i-1)!\rangle, \ \ c_{i}\in F\langle 1, \ldots , a^{s+i}/(s+i)!\rangle .\eqno (2)$$
Really, for $i=1$ it follows from the relation (1) that
$$ [D_{3}, c_{1}D_{1}+d_{1}D_{2}]= \sum _{i=0}^{s}\beta _{i}a^{i}/i!D_{1}= D_{3}(c_{1})D_{1}+(d_{1}a^{s}/s!)D_{1}+D_{3}(d_{1})D_{2}.  \eqno (3) $$
  It follows from (3) that $D_{3}(d_{1})=0$, and since the ideal $FI_{2}$ is abelian,  it holds $D_{1}(d_{1})=D_{2}(d_{1})=0.$  The latter means that $d_{1}\in F=F\langle 1\rangle .$ We also get from (3) that $D_{3}(c_{1})\in F\langle 1, \ldots , a^{s}/s!\rangle $ and obviously it holds $D_{1}(c_{1})=D_{2}(c_{1})=0.$ Then, by Lemma \ref{inclusion1},  $c_{1}\in F\langle 1, a, \ldots , a^{s+1}/(s+1)!\rangle $ and the relations (2) hold for $i=1.$ Assume they   hold for $i-1.$ Let us prove that the relations (2)  hold for $i.$ Using the equalities
$[D_{3}, c_{i}D_{1}+d_{i}D_{2}]=c_{i-1}D_{1}+d_{i-1}D_{2}$   and  $[D_{3}, D_{2}]=a^{s}/s!D_{1}$ we get $D_{3}(d_{i})=d_{i-1}, D_{3}(c_{i})=d_{i}a^{s}/s!+c_{i-1}.$  By the inductive assumption, we have  $d_{i-1}\in F\langle 1, a, \ldots , a^{i-2}/(i-2)!\rangle ,$ hence  $ d_{i}\in F\langle 1, a, \ldots , a^{i-1}/(i-1)!\rangle $ by Lemma \ref{inclusion1}. Analogously, by the inductive assumption it holds  $$c_{i-1}\in F\langle 1, a, \ldots , a^{s+i-1}/(s+i-1)!\rangle $$ and therefore $D_{3}(c_{i})\in F\langle 1, a, \ldots , a^{s+i-1}/(s+i-1)!\rangle .$ Since $D_{1}(c_{i})=D_{2}(c_{i})=0$ we get  by Lemma \ref{inclusion1} that  $c_{i}\in F\langle 1, a, \ldots , a^{s+i}/(s+i)!\rangle .$ But then we have inclusion $$FI_{2}\subseteq F\langle D_{1}, aD_{1}, \ldots , (a^{s+k}/(s+k)!)D_{1}, D_{2}, aD_{2}, \ldots , (a^{k}/k!)D_{2}\rangle .$$ The last subalgebra of the Lie algebra $W(R)$  is contained in the subalgebra of the form $$F\langle D_{1}, aD_{1}, \ldots , (a^{s+k}/(s+k)!)D_{1}, D_{2}, aD_{2}, \ldots , (a^{s+k}/(s+k)!)D_{2}\rangle . $$
But then the Lie algebra $L$ is contained in the subalgebra $\widetilde L$  from the conditions  Lemma \ref{rk2}.

Let now $s=0.$ Then  $FI_{1}=FD_{1}$ and without loss of generality we may assume that $[D_3, D_2]=D_1.$ Repeating the above considerations we can build a Jordan basis $\{ (c_iD_1+d_iD_2)+FI_1, i=1, \ldots , k\}$ of the quotient algebra $FI_2/FI_1$ with $[D_3, c_iD_1+d_iD_2]=c_{i-1}D_1+d_{i-1}D_2, i=2, \ldots , k$ and $[D_3, c_1D_1+d_1D_2]=\alpha D_1$ for some $\alpha \in F.$ It follows from the last  equality that $D_3(d_1)=0$ and taking into account the equalities $D_1(d_1)=0$ and $D_2(d_1)=0$ we see that $d_1\in F.$  Since $a_1D_1+d_1D_2\not\in FI_1,$  we have  $d_1\not=0.$
By conditions of the lemma, $\dim _{F}FL>3,$  so we have $k\geq 2$ and the relation $[D_3,  c_2D_1+d_2D_2]=c_1D_1+d_1D_2$ implies the equality $D_3(d_2)=d_1.$ But then $D_3(d_2d_1^{-1})=1$ and multiplying all the elements of  the Jordan basis considered above by $d_{1}^{-1}$ we may assume that $D_3(d_2)=1.$
 Denoting  $a=d_2$  and repeating the considerations from the subcase $s>0$ we see that the Lie algebra $L$ is contained in the subalgebra $\widetilde L$  from the conditions  of Lemma \ref{rk2}.

{\underline{Case 2.}} The ideal $I_{2}$ is nonabelian. We may assume without loss of generality that $I_{1}$ coincides with its centralizer in $L$, i.e.  $C_{L}(I_{1})=I_{1}.$ Really, let $C_{L}(I_{1})\supset I_{1}$  with strong containment. Choose a one-dimensional (central) ideal $(D_{4}+I_{1})/I_{1}$ in the ideal $C_{L}(I_{1})/I_{1}$ of the quotient algebra $L/I_{1}.$ Then $I_{4}:=(RD_{1}+RD_{4})\cap L$ is an abelian ideal of rank $2$ of the algebra $L$ and $\dim _{F}FL/FI_{4}=1$ by Lemma 5 from \cite{MP1}. Thus the problem is reduced to the case 1 (one should take $FI_{4}$ instead of $FI_{2}$).  So, we assume that $C_{L}(I_{1})=I_{1}.$  It follows from this equality that   $C_{FL}(FI_{1})=FI_{1}.$

As in the case 1 we write $FL=FI_{2}+FD_{3}$ and $[D_{3}, D_{2}]=rD_{1}$ for some $r\in R.$ Since the ideal $FI_{1}$ is abelian, the linear operator ${\rm ad} [D_{3}, D_{2}]={\rm ad}( rD_{1})$ acts trivially on the vector space $FI_{1},$ and therefore the linear operators ${\rm ad} D_{2}$ and ${\rm ad} D_{3}$ commute on  $FI_{1}.$  Denote by $M_{2}$ the kernel ${\rm Ker} ({\rm ad} D_{2})$ on the $F$-space $FI_{1}.$ It is obvious that $M_{2}$ is an abelian subalgebra of $FI_{1}$  and $M_{2}$ is invariant under the  action of ${\rm ad} D_{3}.$ Since $[D_{1}, M_{2}]=[D_{2}, M_{2}]=0$ the linear operator ${\rm ad} D_{3}$ has on the  $F$-space $FI_{1}$ the kernel of dimension $1$ (in other case the center of the Lie algebra $FL$ would have dimension $\geq 2$ over $F$ which contradicts  our assumption). Using Lemma \ref{nilfewdim} one can easily show that
$$ M_{2}=F\langle  D_{1}, aD_{1}, \ldots , (a^{k}/k!)D_{1}\rangle $$
 for some  $a\in R$ with $D_{1}(a)=0, D_{2}(a)=0, D_{3}(a)=1$ (if $k=0,$  then put $M_{2}=FD_{1}).$
Further denote $M_{3}={\rm Ker} ({\rm ad} D_{3})$ on the vector space $FI_{1}.$ As above one can prove that $M_{3}$ is invariant under action of ${\rm ad} D_{2},$ this linear operator has one-dimensional  kernel on $M_{3},$ and
$$M_{3}=F\langle D_{1}, bD_{1}, \ldots , (b^{m}/m!)D_{1}\rangle$$
for some $b\in R$ with  $D_{1}(b)=D_{3}(b)=0$ and $D_{2}(b)=1$
(if $m=0$ put $M_{3}=FD_{1}$).

Take now any element $cD_{1}$ of the ideal $FI_{1}, c\in R.$  Since the linear operators ${\rm ad} D_{2}$ and ${\rm ad} D_{3}$ act nilpotently on $FI_{1},$ there exist the least positive integers $k_{0} $ and $m_{0}$ (depending on the element $cD_{1}$)  such that $({\rm ad} D_{2})^{k_{0}}(cD_{1})=0, ({\rm ad} D_{3})^{m_{0}}(cD_{1})=0.$ Let us show by induction on $s=m_{0}+k_{0}$ that the element $cD_{1}$ is a linear combination (with coefficients from $F$) of elements of the form $\frac{a^{i}b^{j}}{i! j!}D_1\in W(A)$ for some $0\leq i\leq k_{0}-1, 0\leq j\leq m_{0}-1$ (note that the elements  $\frac{a^{i}b^{j}}{i! j!}D_1$ can be outside of $FI_{1}$).
If $s=2$ (obviously $s\geq 2$), then we must only consider the case $m_{0}=1, k_{0}=1.$ In this case, we have   $[D_{3}, cD_{1}]=0, [D_{2}, cD_{1}]=0.$ These equalities imply  that $cD_{1}\in Z(FL)=FD_{1}$ and all is done. Let $s\geq 3.$ The element $[D_{2}, cD_{1}]$ can be written by the inductive assumption  in the form
$$ [D_{2}, cD_{1}]=\sum _{i=0}^{k_{0}-2}\sum _{j=0}^{m_{0}-1}\gamma _{ij}\frac{a^{i}b^{j}}{i! j!}D_{1} \  \mbox{for some } \ \gamma _{ij}\in F. $$
Analogously we get
$$ [D_{3}, cD_{1}]=\sum _{i=0}^{k_{0}-1}\sum _{j=0}^{m_{0}-2}\delta _{ij}\frac{a^{i}b^{j}}{i! j!}D_{1} \  \mbox{for some } \ \delta _{ij}\in F.$$
It follows from the previous two equalities that
$$ D_{2}(c)=\sum _{i=0}^{k_{0}-2}\sum _{j=0}^{m_{0}-1}\gamma _{ij}\frac{a^{i}b^{j}}{i! j!}, \ \  \ D_{3}(c)=\sum _{i=0}^{k_{0}-1}\sum _{j=0}^{m_{0}-2}\delta _{ij}\frac{a^{i}b^{j}}{i! j!}.$$
Note that $[D_{3}, D_{2}](c)=rD_{1}(c)=0$ and therefore by Lemma \ref{inclusion2} $c\in F\langle \frac{a^{i}b^{j}}{i! j!}, \ 0\leq i\leq k_{0}-1, 0\leq j\leq m_{0}-1 \rangle .$ Since $cD_{1}$ is arbitrarily chosen we have $FI_{1}\subseteq F\langle \frac{a^{i}b^{j}}{i! j!}D_{1}, 0\leq i\leq k_{0}-1, 0\leq j\leq m_{0}-1\rangle .$ One can straightforwardly  check that $k_{0}\leq k,$ where $k=\dim M_{2}-1$ and analogously $m_{0}\leq m= \dim M_{3}-1.$
Let, for example,  $m\geq n.$ Then $FI_{1}\subseteq F\langle \frac{a^{i}b^{j}}{i! j!}D_{1}, 0\leq i, j\leq m\rangle .$

Further, by the above proven, the linear operator ${\rm ad} D_{3}$ on the vector space $FI_{2}/FI_{1}$ has a matrix   in a  basis in the form of a single Jordan block.  This  basis can be chosen in the form
$ (u_{1}D_{1}+v_{1}D_{2})+FI_{1}, \ldots , (u_{t}D_{1}+v_{t}D_{2})+FI_{1}$ such that
$$[D_{3}, u_{i}D_{1}+v_{i}D_{2}]=u_{i-1}D_{1}+v_{i-1}D_{2}, i\geq 2 ,  \ [D_{3}, u_{1}D_{1}+v_{1}D_{2}]=fD_{1} \eqno (4) $$
 for some element $f,$ \  $ f\in F\langle \frac{a^{i}b^{j}}{i! j!}, 0\leq i, j\leq m\rangle .$ Let us show by induction on $s$ that
$$ u_{s}\in F\langle \frac{a^{i}b^{j}}{i! j!}, \ 0\leq i,  j\leq m+s \rangle ,  v_{s}\in F\langle 1, \ldots , a^{s-1}/(s-1)!\rangle  .$$
If $s=1,$ then the equalities
$$[D_{3}, u_{1}D_{1}+v_{1}D_{2}]=fD_{1}=D_{3}(u_{1})D_{1}+D_{3}(v_{1})D_{2}+v_{1}rD_{1} \eqno (5) $$
imply  $D_{3}(v_{1})=0$ (let us recall here that $[D_{3}, D_{2}]=rD_{1}$). Taking into account the relations $[D_{1}, u_{1}D_{1}+v_{1}D_{2}]=0$ and $[D_{2}, u_{1}D_{1}+v_{1}D_{2}]\in FI_{1}$ we obtain that $v_{1}\in \cap _{i=1}^{3}R^{D_{i}}=F,$ that is $v_{1}\in F\langle 1\rangle .$
It follows from the relations (4) that $$D_{3}(u_{1})+v_1r\in F\langle \{\frac{a^{i}b^{j}}{i! j! }\}, \ 0\leq i, j\leq m. \rangle .$$ Analogously the inclusion $[D_{2}, u_{1}D_{1}+v_{1}D_{2}]\in FI_{1}$ implies the relation
$$D_{2}(u_{1})\in  F\langle \{\frac{a^{i}b^{j}}{i! j! }\}, \ 0\leq i, j\leq m. \rangle .$$
Since $[D_{3}, D_{2}]=rD_{1}$ and $rD_{1}(u_{1})=0,$ we see (using  Lemma \ref{inclusion2}) that $$u_{1}\in F\langle \{\frac{a^{i}b^{j}}{i! j! }\}, \ 0\leq i, j\leq m+1. \rangle .$$
By inductive assumption, we have
$$u_{s-1}\in F\langle \{\frac{a^{i}b^{j}}{i! j! }\}, \ 0\leq i, j\leq m+s-1 \rangle , \ \ v_{s-1}\in F\langle 1, \ldots , \frac{a^{s-2}}{(s-2)!}\rangle .$$
Note that the relations (4) imply the equalities  $D_{3}(u_{s})=u_{s-1}-rv_{s}, \ D_{3}(v_{s})=v_{s-1}$ (here $[D_{3}, D_{2}]=rD_{1}$). Analogously it follows from the relation $[D_{2}, u_{s}D_{1}+v_{s}D_{2}]\in FI_{1}$  that $$D_{2}(v_{s})=0, D_{2}(u_{s})\in  F\langle \{\frac{a^{i}b^{j}}{i! j! }D_{1}\}, \ 0\leq i, j\leq m \rangle .$$
 Since $D_{1}\in Z(L),$ we have the equalities   $D_{1}(u_{s})=D_{1}(v_{s})=0.$   Therefore we get by Lemma \ref{inclusion1} that $v_{s}\in F\langle 1, \ldots , a^{s-1}/(s-1)!\rangle .$ By Lemma \ref{inclusion2}, $u_{s}\in F\langle \{\frac{a^{i}b^{j}}{i! j! }\}, \ 0\leq i, j\leq m+s \rangle $ (since $D_{3}(u_{s})\in F\langle \{\frac{a^{i}b^{j}}{i! j! }\}, \ 0\leq i, j\leq m+s-1 \rangle $ by the relations (4)).
Since $rD_{1}\in FI_{1},$ we have  $rv_{s}\in F\langle \{\frac{a^{i}b^{j}}{i! j! }\}, \ 0\leq i, j\leq m+s-1 \rangle .$
But then by Lemma \ref{inclusion2} $u_{s}\in F\langle \{\frac{a^{i}b^{j}}{i! j! }\}, \ 0\leq i, j\leq m+s \rangle .$
So, we have proved that the Lie algebra $L$ is contained    in the subalgebra $\overline L$ from the conditions of the lemma. To finish with the proof we must prove that  the element $D_{2}$  can be chosen in $W(A)$ in such a way that
$[D_3, D_2]=0.$ Take the element $D_{2}-r_{0}D_{1}$ instead $D_2,$ where the element $r_0$ is obtained from $r$ by formal integration on variable $a$ (recall that $r\in F\langle \{\frac{a^{i}b^{j}}{i! j!} \}, \ 0\leq i, j\leq m \rangle $). Then $[D_{3}, D_{2}]=0.$ The proof is complete.
 \end{proof}

\begin{theorem}\label{th1}
Let $\mathbb K$ be a field of characteristic zero, $A$ an integral $\mathbb K$-domain with fraction field $R.$
 Denote by $W(A)$ the subalgebra $R{\rm Der} _{\mathbb K}A$  of the Lie algebra ${\rm Der} _{\mathbb K}R.$  Let $L$ be a nilpotent subalgebra of rank $3$ over $R$ from $W(A)$ and $F=R^{L}$ its field of constants. If $\dim _{F}FL\geq 4,$ then there exist  integers $n\geq 0,  m\geq 0,$  elements $  D_{1}, D_{2}, D_{3}\in FL$ such that $ [D_{i}, D_{j}]=0, i,j=1,2,3$ and   the Lie algebra $FL$ is contained in one of the following subalgebras of the Lie algebra  $W(A):$

1) $ L_{1}=F\langle D_{3}, D_{1}, aD_{1}, \ldots, (a^{n}/n!)D_{1}, D_{2}, aD_{2}, \ldots , (a^{n}/n!)D_{2}\rangle ,$\\
where $a\in  R$ is  such that $ D_{1}(a)=D_{2}(a)=0, D_{3}(a)=1.$

2) $L_{2}=F\langle D_{3}, D_{2}, aD_{2}, \ldots , (a^{n}/n!)D_{2}, \{ \frac{a^{i}b^{j}}{i! j!}D_{1}\} , 0\leq i, j\leq m\rangle $ where $a, b\in R $ are such that $D_{1}(a)=D_{2}(a)=0, D_{3}(a)=1, D_{1}(b)=D_{3}(b)=0, D_{2}(b)=1.$

\end{theorem}

As a corollary we get the next characterization of nilpotent Lie algebras of rank $\leq 3$ from the Lie algebra $W(A).$
\begin{theorem}\label{th2}
Under conditions of Theorem 1, every nilpotent subalgebra  $L$  of rank $k\leq 3$ over $R$ from the Lie algebra $W(A)$  is isomorphic to a finite dimensional subalgebra of the triangular Lie algebra $u_{k}(F).$
\end{theorem}
\begin{proof}
If $k=1$ then the Lie algebra $FL$ is one-dimensional over $F$ and therefore is isomorphic to $u_{1}(F)=F\frac{\partial}{\partial x_1}.$ In the case $k=2,$ the Lie algebra $FL$ is (by Lemma \ref{MP-4})  of the form
 $$FL=F\langle D_1, aD_1, \ldots , \frac{a^k}{k!} D_{1}, D_{2}\rangle , \  k\geq 0  $$
 $$\ (\mbox{if} \ k=0, \ \mbox{then put} \ FL=F\langle D_{1}, D_{2}\rangle ),$$ where $[D_1, D_2]=0$, $D_1(a)=0$, $D_2(a)=1.$ The Lie algebra $FL$ is isomorphic to a suitable  subalgebra of the triangular Lie algebra  $u_{2}(F)=\{ f(x_2)\frac{\partial}{\partial x_1}+F\frac{\partial}{\partial x_2}\}:$ the correspondence $ D_{i}\mapsto \frac{\partial}{\partial x_i}, i=1, 2$ and $a \mapsto x_2$ can be extended to an isomorphism between $FL$ and a subalgebra of $u_{2}(F).$ Let now $k=3.$
 If $\dim_{F}FL=3,$ then $FL$ is either abelian or has a basis $D_1, D_2, D_3$ with multiplication rule $[D_3, D_2]=D_1, [D_2, D_1]=[D_3, D_1]=0.$   In the first case, $FL$ is isomorphic to the subalgebra $F\langle \frac{\partial}{\partial x_1}, \frac{\partial}{\partial x_2}, \frac{\partial}{\partial x_3}\rangle$, in the second case it is isomorphic to the subalgebra $F\langle \frac{\partial}{\partial x_1}, x_{3}\frac{\partial}{\partial x_1}+\frac{\partial}{\partial x_2}, \frac{\partial}{\partial x_3}\rangle $ of the triangular Lie algebra $u_{3}(F).$

Let now $\dim _{F}FL\geq 4.$  The Lie algebra $FL$ is contained (by Theorem \ref{th1}) in one of the Lie algebras $L_1$ or $L_2$ from the statement of  that theorem. Note that the Lie algebra $L_1$ is isomorphic to the subalgebra $\overline {L}_1$ of the Lie algebra $u_{3}(F)$ of the form $$\overline L_{1}=F\langle \frac{\partial}{\partial x_3}, \frac{\partial}{\partial x_1},  \ldots, (x_{3}^{n}/n!)\frac{\partial}{\partial x_1}, \frac{\partial}{\partial x_2},  \ldots , (x_{3}^{n}/n!)\frac{\partial}{\partial x_2}\rangle ,$$
Analogously the Lie algebra  $L_2$ is isomorphic the the subalgebra $\overline {L}_2$ of $u_{3}(F)$ of the form
$$\overline L_{2}=F\langle \frac{\partial}{\partial x_3}, \frac{\partial}{\partial x_2},  \ldots, (x_{3}^{n}/n!)\frac{\partial}{\partial x_2}, \{ \frac{x_2^{i}x_3^{j}}{i!j!}\frac{\partial}{\partial x_1}\}, 0\leq i, j\leq m
\rangle .$$
\end{proof}
\begin{corollary}
Let $L$ be a nilpotent subalgebra of the Lie algebra $W_{3}(\mathbb K)={\rm Der} (\mathbb K[x_1, x_2, x_3])$ and $F$ the field of constants for the Lie algebra  $L$ in the field  $\mathbb K(x_, x_2, x_3).$ Then the Lie algebra $FL$ (over the field $F$) is isomorphic to a finite dimensional subalgebra of the triangular Lie algebra $u_{3}(F).$
\end{corollary}

\begin{remark}
If $L$ is a nilpotent subalgebra of rank $3$ over $R$ from the Lie algebra $W_{3}(\mathbb K),$ then $L$ being isomorphic to a subalgebra of the triangular Lie algebra $u_{3}(\mathbb K)$  can be not conjugated (by an automorphism of $W_{3}(\mathbb K)$) with any subalgebra of $u_{3}(\mathbb K).$  Indeed, the subalgebra $L=\mathbb K\langle x_1\frac{\partial}{\partial x_1}, x_2\frac{\partial}{\partial x_2}, x_3\frac{\partial}{\partial x_3}\rangle$ is nilpotent but not conjugated with any subalgebra of $u_{3}(\mathbb K)$ ($L$ is selfnormalized in  $W_{3}(\mathbb K)$, but any finite dimensional subalgebra of $u_{3}(\mathbb K)$ is not,  because of locally nilpotency of the Lie algebra $u_{3}(\mathbb K)$).
\end{remark}

The author is grateful to V.~Bavula and V.M.Bondarenko for useful discussions and advice.

%%%%%%%%%%%%%%%%%%%%%%%%%%%%%%%%%%%%%%%%%%

%
\end{document}